\documentclass[12pt]{amsart}

\usepackage{
amsfonts,
latexsym,
amssymb,
enumerate,
verbatim,
mathrsfs,
}

\newcommand{\labbel}{\label}

\newtheorem{theorem}{Theorem}[section]

\newtheorem{proposition}[theorem]{Proposition} 
 
\newtheorem{corollary}[theorem]{Corollary}

\newtheorem*{theorem*}{Theorem}
\newtheorem*{corollary*}{Corollary}

\theoremstyle{definition}

\theoremstyle{remark}
\newtheorem{remark}[theorem]{Remark}

\newtheorem*{acknowledgement}{Acknowledgement}

 \allowbreak

 \allowdisplaybreaks[1]

\begin{document}
 
\title[Relation identities, congruence modularity]
{Relation identities equivalent\\to congruence modularity}

\author{Paolo Lipparini} 
\address{ 
Dipartimento Relativo di Matematica\\Viale della  Ricerca
 Scientifica\\Universit\`a di Roma ``Tor Vergata'' 
\\I-00133 ROME ITALY}
\urladdr{http://www.mat.uniroma2.it/\textasciitilde lipparin}


\thanks{Work performed under the auspices of G.N.S.A.G.A. Work 
partially supported by PRIN 2012 ``Logica, Modelli e Insiemi''}

\begin{abstract}
We present some identities 
dealing with reflexive and admissible relations and which, through
a variety, are equivalent to congruence modularity.
\end{abstract}

\maketitle

\section{Introduction} \labbel{int} 

Congruences and congruence identities
have proved to be fundamental notions in universal algebra.
See, e.~g., J{\'o}nsson \cite{cd} for an introduction.
It has been observed that sometimes 
reflexive and admissible relations play an important role
even when the main concern  are congruences.
See, e.~g.,  \cite[p. 370]{cd} and
Tschantz \cite{T}, just to mention some. 
Tolerances, too, have sometimes proved useful,
see, for example, 
Chajda \cite{cha}, 
 Cz\'edli, Horv\'ath, and  Lipparini \cite{CHL},
Kearnes and Kiss \cite{kk}, 
Lipparini \cite{contol}
and further references there. 
Many identities equivalent to 
congruence modularity are known, 
e.~g., the quoted \cite{CHL,cd,T},
Freese and J{\'o}nsson  \cite{FJ}, 
Gumm \cite{G1,G2} and further references in the quoted papers.
We shall describe here some identities 
which are equivalent to congruence modularity
but are expressed also in terms of reflexive and admissible relations.
A sample  of the identities we have found 
is given in the following theorem, but 
first we need to introduce some notations.

Juxtaposition denotes intersection,
$ \circ $ denotes composition of binary relations
and, for $R$ a relation,  $R ^\smallsmile $ denotes the \emph{converse}
of $R$, that is, $b \mathrel {R ^\smallsmile }  a $
holds if and only if $a \mathrel R b  $.
By $R^*$ we denote the transitive  
closure of $R$ and 
 $ \overline{R} $ denotes the 
smallest reflexive and admissible relations 
containing  $R$
(of course, this is dependent on the algebra we are working in).
Recall that a \emph{tolerance} $\Theta$ 
is a reflexive, symmetric and admissible relation.
For simplicity, at first reading, the reader might always 
take all tolerances here to be congruences.

We say that a relation identity $\varepsilon$ 
\emph{holds} in some variety $\mathcal V$ if,
for every algebra $\mathbf A \in \mathcal V$,
the identity $\varepsilon$ holds for all reflexive and admissible relations 
of $\mathbf A$. Some variables in $\varepsilon$ might
be required to vary among 
tolerances or congruences;
formally, this makes no difference, since if
$R$ is a variable for reflexive and admissible relations,
then, say, $ (R \circ R ^\smallsmile )^* $ can be considered as (or 
substituted for) a variable for congruences.

Notice that an inclusions such as 
$\iota \subseteq \iota'$ 
can be considered as an identity, since it is equivalent to
$\iota = \iota \iota'$.

\begin{theorem} \labbel{d}
For every variety, each of the following 
identities  is equivalent to congruence modularity.
\begin{align*}
\labbel{11a}  \tag{1.1}
\Theta (S \circ   S)
&\subseteq 
( \Theta  S) ^*   
\\
 \labbel{11b}\tag{1.2}
\Theta S^*
&\subseteq 
( \Theta  S) ^*
\\
 \labbel{11c}\tag{1.3}
\Theta (S \circ S ^\smallsmile )
&\subseteq 
(\Theta  S \circ  \Theta S ^\smallsmile)^*
\\
 \labbel{11d}\tag{1.4}
\Theta (S \circ  T )^*
& \subseteq 
\Theta (\overline{S \cup T}) \circ 
 (\Theta  S \circ   \Theta  T )^*   
\\
 \labbel{11e}\tag{1.5}
\Theta (S \circ  T )
& \subseteq 
\Theta (\overline{S ^\smallsmile  \cup T}) \circ 
 (\Theta  S \circ   \Theta  T )^*   
 \end{align*}    
where  $S$ and $T$ vary among 
reflexive and admissible relations and     
$\Theta$ can be equivalently taken to vary  among tolerances
or  congruences.
 \end{theorem}

The result is quite curious since, by minimal 
variations on the above identities,
we get identities which are not equivalent to congruence modularity.
For example, if we ``merge'' \eqref{11c} and  \eqref{11d}
as $\Theta (S \circ T ^\smallsmile )
\subseteq 
(\Theta  S \circ  \Theta T ^\smallsmile)^*$, 
we get an identity equivalent to congruence distributivity,
hence strictly stronger than modularity.
As another example, 
the variation 
$\Theta (S \circ S)
\subseteq 
( \Theta S ^\smallsmile)^*$ of \eqref{11a}, too, 
is strictly stronger than modularity, since it implies $m$-permutability
for some $m$.

The identities in Theorem \ref{d} are special cases of
the identities  (A1),  (B1), (C1) and (D1)
in Corollary \ref{tut}, which shall be proved below.

\section{A strong  identity for relations} \labbel{anid}

H.-P. Gumm \cite{G1,G2} 
 provided a  characterization of congruence modular
varieties
by means of the existence of certain terms;
we shall not need the explicit description of Gumm terms in what follows.
 Kazda,  Kozik,  McKenzie and Moore
\cite{kkmm} 
showed that a variety has Gumm terms if and only if 
it has \emph{directed Gumm terms},
that is, terms $p, j_1, \dots, j_k$ satisfying the following set of identities,
for some $k$.

\begin{align}
\labbel{1} 
\tag{DG1}
x & = p(x,z,z)
\\
\labbel{2} 
\tag{DG2}
 p(x,x,z) &= j_1(x,x,z) 
  \\
\labbel{3} 
\tag{DG3}
  x &=j_i(x,y, x), 
&& \text{ for } 
1 \leq i \leq k,
 \\ 
\labbel{4} 
\tag{DG4}
  j_{i}(x,z,z) &=
j_{i+1}(x,x,z)  && \text{ for \ } 
1 \leq i < k
 \\ 
\labbel{5} 
\tag{DG5}
j_{k}(x,y,z) &=z  
\end{align}   

In particular, by the mentioned results,
a variety
is congruence modular if and only if 
it has  directed Gumm terms, for some $k$.
Notice that we have given the definition 
of directed J{\'o}nsson terms in the reversed order,
in comparison with \cite{kkmm}.
However, the two definitions are obviously equivalent:
just simultaneously reverse both the order of variables and the order of terms.

Recall the notations introduced right
before Theorem \ref{d};
in particular, recall that juxtaposition denotes
intersection. 
Furthermore, we let $S \circ _m T$  denote
$ S \circ T \circ S \dots$
with $m$ factors, that is, 
with $m-1$ occurrences of $\circ$.
 Moreover,  $R^h $ is $  R \circ R \circ R \dots$
with $h$ factors, that is, 
$R^h = R \circ _h R$. 
 We let
$S+T$ denote  $\bigcup _{m \in \mathbb N} S \circ_m T$; in particular,
for $\alpha$ and $\beta$ congruences,
$ \alpha + \beta $ is the \emph{join} in the congruence lattice.
Notice that the set of all reflexive and admissible relations on some 
algebra also forms a lattice, but in this case the join of
$S$ and $T$ is 
$ \overline{S \cup T} $.
We shall frequently use the fact that
$ \overline{S \cup T}  \subseteq S \circ T$,
for reflexive and admissible relations $S$ and $T$.
Notice also that, in the above notations,
for a reflexive relation $R$, we have
$R^*= R + R$.
If $R$ is a reflexive and admissible relation, let 
$\Theta_R$ be the smallest tolerance containing $R$, 
that is, $\Theta = \overline{R \cup R ^\smallsmile } $.

\begin{theorem} \labbel{t}
If a  variety $\mathcal V$ 
has $k+1$ directed Gumm terms 
$p, j_1, \dots, \allowbreak  j_k$, with 
$k \geq 2$, then, for every 
natural number $\ell \geq 1$,   
 $\mathcal V$ satisfies the following identities
\begin{multline}   
\labbel{turt}
R (V \circ W)(S_1 \circ S_2 \circ \dots \circ S_ \ell)
\subseteq 
\\
R  ( \overline{V \cup W}) \circ 
( \Theta_R  S_1 \circ \Theta_R   S_2 \circ \dots \circ \Theta_R  S_ \ell)  ^{2k-3} 
\end{multline} 
\begin{multline}    
\labbel{turtt}
R (V \circ W)(S_1 \circ   S_2 \circ \dots \circ S_ \ell)
\subseteq 
\\
R R ^\smallsmile  ( \overline{V  ^\smallsmile \cup W}) \circ 
( \Theta_R  S_1 \circ \Theta_R   S_2  \circ \dots \circ \Theta_R  S_ \ell) ^{k-1} 
\end{multline}
where 
$R$,  $V$, $W$, $S_1$ \dots\ vary among
reflexive and admissible relations.
 \end{theorem} 

\begin{proof}
Suppose that $\mathbf A$ is an algebra belonging to $\mathcal V$
and that in $\mathbf A$ we have
$(a,c ) \in R (V \circ W)(S_1 \circ S_2 \circ \dots \circ S_ \ell)$,
for certain reflexive and admissible relations $R$, $V, \dots$  
Then
$ a \mathrel R  c$,
$ a \mathrel { V } b \mathrel {W } c $
 and 
$ a = a_0 \mathrel { S_1} a_1 \mathrel {S_2} a_2 \dots a _{ \ell-1} 
\mathrel { S_ \ell} a_ \ell = c  $,
for certain elements 
$ b, a_1, a_2, \dots$\  
In order to prove  \eqref{turt}, let us compute 
\begin{gather*}   
a= p(a,p(a \text{\boldmath$a$}b), p(aa\text{\boldmath$b$}))
\mathrel { \overline{V \cup W} } 
p(a,p(a
\text{\boldmath$b$}
b), p(aa\text{\boldmath$c$})) = p(a,a, p(aac)),
\\
a= 
p(a,a, a)=
p(a,a,p(aa \text{\boldmath$a$}))
\mathrel { R  } 
p(a,a, p(aa \text{\boldmath$c$})),
 \end{gather*}
where elements in bold are those moved by $V$, $W$ or $R$
and we have used \eqref{1}.
Moreover,
 $p(a, a, p(aac)) =
j_1(a, a, j_1(aac))$,
by \eqref{2}, 
 hence
\begin{equation}\labbel{a}  
a \mathrel {R(\overline{V \cup W}) }  j_1(a, a, j_1(aac))
  \end{equation}   
For $h=0, \dots, \ell-1$, 
we have 
\begin{equation*}    
\text{$j_1(a, \text{\boldmath$a_h$}, j_1(a\text{\boldmath$a_h$}c) )
\mathrel {S_h}
 j_1(a, \text{\boldmath$a_{h+1}$}, j_1(a\text{\boldmath$a_{h+1}$}c))$}
 \end{equation*} 
For sake of brevity,
let
$j^*(x,y,z) = j_1(x,y, j_1(xyz))$,
thus 
$j^*$  satisfies $x=j^*(x,y,x)$, by \eqref{3}.
Then
\begin{multline*}
j_1(a, a_h, j_1(a a_h c) )
=
j^*(a,a_h, c)
=
j^*( j^*(aa_{h+1}  \text{\boldmath$a$}),a_h, j^*( \text{\boldmath$c$}a_{h+1} c))
\mathrel \Theta_R  
\\
j^*( j^*(aa_{h+1} \text{\boldmath$c$}),a_h, j^*( \text{\boldmath$a$}a_{h+1} c))=
j^*(a a_{h+1} c) = j_1(a, a_{h+1}, j_1(aa_{h+1}c))
 \end{multline*}   

 Hence 
$j_1(a, a_h, j_1(aa_hc)) \mathrel {\Theta_R S_h} j_1(a, a_{h+1}, j_1(aa_{h+1}c))$,
for $h=0, \dots,  \allowbreak \ell-1$.
Concatenating, and setting
$\Lambda= \Theta_R  S_1 \circ \Theta_R   S_2 \circ \dots \circ \Theta_R  S_ \ell$
we get
$j_1(a, a, j_1(aac)) \mathrel {  \Lambda }  j_1(a, c, j_1(acc)) 
 =  j_1(a, c, j_2(aac))  $, by \eqref{4}.

By similar (and easier) arguments,
we have 
 $j_2(a,a,c) \mathrel {  \Lambda }j_2(a,c,c) =j_3(a,a,c)   $,
hence
$ j_1(a, c, j_2(aac))  \mathrel {  \Lambda }  j_1(a, c, j_2(acc)) 
= j_1(a, c, j_3(aac))  $.
Iterating,
 $ j_1(a, c, j_3(aac))  \mathrel {  \Lambda }  
 j_1(a, c, j_4(aac))  $ \dots \ 
Concatenating again, we get 
\begin{multline} \labbel{b}  
j_1(a, a, j_1(aac)) \mathrel {  \Lambda ^ {k-1} }  j_1(a, c, j_{k-1}(acc)) 
 = 
\\
j_1(a, c, j_k(aac))  =  
j_1(a, c, c)  = j_2(a, a, c)  \mathrel {  \Lambda ^ {k-2} }  
j_{k-1}(a,c,c) =c 
 \end{multline}  
by \eqref{5}. 
Putting together
\eqref{a} and \eqref{b}, we get 
$(a,c) \in R( \overline{V \cup W}) \circ \Lambda  ^{2k-3} $,
thus equation \eqref{turt} is proved.

  The proof of equation \eqref{turtt} 
is much simpler.
We have 
$a=p(a,b,b) \allowbreak \mathrel { \overline{V  ^\smallsmile \cup W}} p(a,a,c)$,
$a=p(a,a,a) \mathrel { R} p(a,a,c) $ and 
 $a=p(a,c,c) \mathrel {R ^\smallsmile }  p(a,a,c)
= j_1(a,a,c) $.
Moreover,
 as above,
 $j_1(a,a,c)  \mathrel  \Lambda ^{k-1}  j_{k-1}(a,c,c) \allowbreak =  c$,
hence  \eqref{turtt} follows.
 \end{proof}

Notice that if $k=1$ in the definition of directed Gumm terms, 
then $p$ is a Maltsev term for congruence permutability.
Since in a congruence permutable variety every
reflexive and admissible relations is a congruence,
all the considerations below will become 
trivial in case $k=1$, so we can always suppose 
$k \geq 2$. Notice that the above arguments 
show that 
$\alpha \circ \beta = \overline{ \alpha \cup \beta } $
holds in a congruence permutable variety. 

\begin{corollary} \labbel{c}
If a  variety $\mathcal V$ 
has $k+1$ directed Gumm terms $p, j_1, \dots, \allowbreak  j_k$, with 
$k \geq 2$, then, for every 
natural number $h \geq 1$,   
 $\mathcal V$ satisfies the identities
\begin{align} \labbel{a1}   
\Theta (S \circ _{2^h}  S)
& \subseteq 
( \Theta  S) ^{q+1} 
\\
\labbel{a2}
R(S \circ _{2^h}  T )
& \subseteq 
R (\overline{S \cup T}) \circ 
 (\Theta_R  S \circ _{q}   \Theta_R  T )
\\
\labbel{a3}
\Theta (S \circ _{2^h}  S ^\smallsmile )
& \subseteq 
 \Theta  S ^\smallsmile  \circ _{r}   \Theta  S 
\end{align}      
where $q=(2 ^{h+1}-2 )(2k-3)$,
$r=1+ (2 ^{h+1}-2 )(k-1)$,
$R$,  $S$, $T$ vary among 
reflexive and admissible relations and     
$\Theta$ varies among tolerances
(or  congruences).
 \end{corollary} 

\begin{proof}
The identity \eqref{a1}  
is the particular case of \eqref{a2}
when $S=T$ and $R= \Theta $, hence we shall go directly to the 
proof of \eqref{a2}.

The case $h=1$ of  \eqref{a2} follows from  
equation \eqref{turt} in Theorem \ref{t},
taking
$\ell = 2$,  $V= S_1= S$ and 
$W = S_2 = T $.
Suppose now that  \eqref{a2}  holds for some $h \geq 1$.
Since  $2^h$ is even,
we have $S \circ _{2^{h+1}}  T =
(S \circ _{2^h}  T) \circ (S \circ _{2^h}  T)$.
Taking
$\ell = 2 ^{h+1} $,  $V= W = S \circ _{2^h}  T$,
$S_1= S_3 = \dots =  S $ and
$S_2= S_4 = \dots =  T $ in equation \eqref{turt},
we get
$R (S \circ _{2^{h+1}}  T )
\subseteq 
R (S \circ _{2^h}  T) \circ 
 (\Theta_R  S \circ _{2^{h+1}(2k-3)}   \Theta_R  T )$,
since $\ell$ is even.
By the inductive assumption,
$R (S \circ _{2^h}  T )
 \subseteq 
R (\overline{S \cup T}) \circ 
 (\Theta_R  S \circ _{q}   \Theta_R  T )
$, hence, noticing that $q$ is even, we get 
 $R (S \circ _{2^{h+1}}  T )
 \subseteq 
R (\overline{S \cup T}) \circ 
 (\Theta_R  S \circ _{q'}   \Theta_R  T )
$, where $q'= q+ 2^{h+1}(2k-3)$.
But  $q'= (2^{h+1}-2)(2k-3)+ 2^{h+1}(2k-3)
= (2^{h+2}-2)(2k-3)$,
what we had to show.

As for the last identity,
in case $h=1$, take $\ell=2$, 
$R= \Theta $, $V= S_1= S$,
$W = S_2 = S ^\smallsmile $
in identity \eqref{turtt} in Theorem \ref{t}, getting
$\Theta (S \circ   S ^\smallsmile )
\subseteq  \Theta  S ^\smallsmile  \circ _{2k-1}   \Theta  S   $.
If the identity \eqref{a3}  holds for some $h \geq 1$, then, 
since  
$S \circ _{2^{h+1}}  S ^\smallsmile  =
(S \circ _{2^h}  S ^\smallsmile ) \circ (S \circ _{2^h}  S ^\smallsmile )$
(here we are using the fact that $2^h$ is even, for $h \geq 1$),
we can apply equation \eqref{turtt} in Theorem \ref{t}
with $\ell= 2 ^{h+1} $, 
$R= \Theta $, $V= W = S \circ _{2^h}  S ^\smallsmile$,
$ S_1= S_3 = \dots =  S $
and 
$ S_2= S_4 = \dots =  S  ^\smallsmile $
getting
$\Theta (S \circ _{2^{h+1}}  S ^\smallsmile )
\subseteq  \Theta (S \circ _{2^{h}}  S ^\smallsmile)
\circ ( \Theta  S  \circ _{2^{h+1}(k-1)}   \Theta  S  ^\smallsmile )  $,
since
$(S \circ _{2^{h}}  S ^\smallsmile) ^\smallsmile 
=  S ^{\smallsmile \smallsmile} \circ _{2^{h}}  S ^\smallsmile 
 = S \circ _{2^{h}}  S ^\smallsmile $,
using the fact that both $2^{h+1}(k-1)$
and  $2^h$ are even.
By the inductive hypothesis,
$ \Theta (S \circ _{2^{h}}  S ^\smallsmile)
\subseteq 
 \Theta  S ^\smallsmile  \circ _r   \Theta  S $,
hence we get 
$ \Theta (S \circ _{2^{h+1}}  S ^\smallsmile)
\subseteq 
 \Theta  S ^\smallsmile  \circ _{r'}   \Theta  S $,
 for $r'= r+ 2^{h+1}(k-1)$ 
noticing that $r$ is odd. 
But 
$r'= r+ 2^{h+1}(k-1)=
1+ (2 ^{h+1}-2 )(k-1) + 2^{h+1}(k-1) =
1+ (2 ^{h+2}-2 )(k-1)  $, what we had to show. 
\end{proof}

\section{Further equivalences and remarks} \labbel{equiv} 

In order to provide a uniform notation for the results in the following corollary,
let $\circ _ {\infty}$ 
be another notation for $+$.
This is justified since
$R \circ _ {\infty} S = R + S  = \bigcup _{n \in \mathbb N} R \circ_n S $.  
Recall that $\Theta_R$ denotes the smallest tolerance containing
the relation $R$. 

\begin{corollary} \labbel{tut} 
For a variety $\mathcal V$ 
and every  $m \geq 2$, possibly
$m= \infty$, each  of the following  identities 
is equivalent to congruence modularity
  \begin{enumerate}    
\item[(A1)]
$\Theta (S \circ_m S) \subseteq  (\Theta S)^*$ equivalently,
$(\Theta (S \circ_m S)) ^*= (\Theta S)^*$ 
\item[(A2)]
$\Theta (S \circ_m S) \subseteq \Theta S+ \Theta S ^\smallsmile $ 
\item[(A3)]
$\Theta (S \circ_m S) \subseteq 
(\Theta  (S ^\smallsmile \circ  S ))^*  $ 
\item[(B1)]
$\Theta (S \circ_m S ^\smallsmile ) \subseteq
\Theta  S+ \Theta  S ^\smallsmile $ equiv.
 $(\Theta (S \circ_m S ^\smallsmile ))^* =
\Theta  S+ \Theta  S ^\smallsmile $
\item[(B2)]
$\Theta (S \circ_m S ^\smallsmile ) \subseteq
 (\Theta ( S ^\smallsmile \circ S))^* $ 
equiv.\,$(\Theta (S \circ_m S ^\smallsmile ))^* =
(\Theta (S ^\smallsmile  \circ  S ))^*$
\item[(C1)]
$R (S \circ_m T) \subseteq 
R ( \overline{S \cup T}) \circ  (\Theta_R   S+ \Theta _R T)$ 
\item[(C2)]
$\Theta (S \circ_m T) \subseteq 
(\Theta ( \overline{S \cup T}))^* $ equiv.
$(\Theta (S \circ_m T))^* =
(\Theta ( \overline{S \cup T}))^* $
\item[(C3)]
$R (S \circ_m T) \subseteq 
R (T \circ \overline{S \cup T}) \circ  (\Theta_R   S+ \Theta_R  T)$ 
\item[(C4)]
$\Theta (S \circ_m T) \subseteq
 (\Theta ( T  \circ S))^* $
equivalently, 
$(\Theta (S \circ_m T))^* = 
(\Theta (T \circ S))^*$ 
\item[(D1)]
$R (S \circ_m T) \subseteq 
R ( \overline{S ^\smallsmile  \cup T  }) \circ  (\Theta_R   S+ \Theta_R  T)$
\item[(D2)]
$R (S \circ_m T) \subseteq 
R ( \overline{S {\cup} T}) 
( \overline{S ^\smallsmile  {\cup} T  }) 
( \overline{S {\cup} T ^\smallsmile })
( \overline{S ^\smallsmile  {\cup} T ^\smallsmile})
 \circ  (\Theta_R   S{+} \Theta_R  T)$
\item[(D3)]
$R (S \circ_m T) \subseteq 
R (\overline{ S {\cup}  S ^\smallsmile {\cup} T {\cup} T ^\smallsmile })
 \circ  
(\Theta_R   S{+} \Theta_R  T {+} \Theta_R  S ^\smallsmile  {+} \Theta_R  T ^\smallsmile )$
\item[(D4)]
$\Theta (S \circ_m T) \subseteq 
\Theta (T \circ S) + \Theta (T \circ T ^\smallsmile ) +
 \Theta (S ^\smallsmile \circ S) + \Theta (S ^\smallsmile \circ T )
+ \Theta (S ^\smallsmile \circ T ^\smallsmile ) +
\Theta (T ^\smallsmile  \circ S) + \Theta (T ^\smallsmile \circ T  )$
\item[(D5)]
$\Theta (S \circ_m T) \subseteq 
\Theta ((T+ T ^\smallsmile  ) \circ S) +
 \Theta (S ^\smallsmile \circ S) + \Theta (S ^\smallsmile \circ (T+ T ^\smallsmile  ) )$
  \end{enumerate}
where $S$, $T$ vary among reflexive and admissible relations,
$\Theta$ can be equivalently taken to vary either 
among congruences or  among tolerances
and $R$  can be equivalently taken to vary either 
among congruences or  reflexive and admissible relations.
\end{corollary}

 \begin{proof}
If one of the above conditions holds when $\Theta$
varies among  tolerances, then it obviously 
holds when $\Theta$ varies among congruences.
A similar observation applies to  $R$.
Moreover, in each line with two conditions, both conditions are obviously 
equivalent, since $^*$ is a monotone and idempotent operator.
In (B2), if $m=2$, in order to get the right-hand identity,
use the left-hand identity twice,
both as it stands  and with $S ^\smallsmile $  in 
place $S$. A similar remark applies to (C4).

By considering congruences 
$\alpha$, $\beta$ and $\gamma$,
 taking $R=\Theta= \alpha $, 
 $S= \beta \circ \alpha \gamma $ 
and $T = \beta $ 
in any one of the above identities,
we get an identity of the form
$ \alpha ( \beta \circ _n \alpha \gamma )
\subseteq
\alpha \beta + \alpha \gamma $,
for some $n \geq 3$ (here it is fundamental to assume that $m \geq 2$).
The most involved case is  (D3): notice that
both $\alpha \gamma \circ \beta \subseteq
 \alpha \gamma \circ \beta \circ  \alpha \gamma $
and  $ \beta \circ \alpha \gamma  \subseteq 
\alpha \gamma \circ \beta \circ \alpha \gamma $,
hence
$ \overline{S \cup S ^\smallsmile  \cup T \cup T ^\smallsmile  }
 \subseteq \alpha \gamma \circ \beta \circ  \alpha \gamma$,
hence
 $ \alpha ( \overline{S \cup S ^\smallsmile  \cup T \cup T ^\smallsmile  })
 \subseteq \alpha (\alpha \gamma \circ \beta \circ  \alpha \gamma)
=\alpha \gamma \circ \alpha \beta \circ  \alpha \gamma$.
Since, for $n \geq 3$, obviously
$\alpha ( \beta \circ \alpha \gamma \circ \beta )
\subseteq  \alpha ( \beta \circ _n \alpha \gamma )$,
then from 
$ \alpha ( \beta \circ _n \alpha \gamma )
\subseteq
\alpha \beta + \alpha \gamma $ we get
$\alpha ( \beta \circ \alpha \gamma \circ \beta )
\subseteq 
\alpha \beta + \alpha \gamma $.
Through a variety, this condition implies
congruence modularity 
by Day \cite{D}.

Hence it remains to show that congruence modularity 
implies each of the identities in the corollary,
in the stronger form in which
$\Theta$ varies among tolerances and $R$
 varies among reflexive and admissible relations.
By the mentioned results
from \cite{G1,kkmm},
we can assume that $\mathcal V$ has
directed Gumm terms, for some $k$.  
Then, for every finite $m$,  Condition (C1) follows from 
equation \eqref{a2} in Corollary \ref{c}.
Of course, if (C1) holds for every finite $m$,
then it holds also for $m= \infty$.  
All the conditions except (B1), (D1) and (D2)
are consequences of (C1), by the  obvious monotonicity
properties 
of the operators present in the identities. 
(B1) is a consequence of equation \eqref{a3} in
Corollary \ref{c}. 

In order to prove (D1), first notice that,
by (C1), we have 
$R(S \circ_m T) \subseteq R (S \circ  T  ) \circ  (\Theta_R   S+ \Theta_R  T)$.
By taking $\ell=2$, 
 $V=S_1=S$  
and $W=S_2=T$ in 
equation \eqref{turtt} in Theorem \ref{t},
we get  $ R (S \circ  T  ) \subseteq 
\Theta ( \overline{S ^\smallsmile  \cup T  }) \circ  (\Theta   S+ \Theta  T)$. 
Putting together the above identities we get (D1).

The proof of the stronger (D2) is slightly more involved. 
By (C1), we have 
$R (S \circ_m T) \subseteq 
R ( \overline{S \cup T } ) \circ  (\Theta_R   S+ \Theta_R  T)
=
R ( \overline{S \cup T } ) (S \circ T)\circ  (\Theta_R   S+ \Theta_R  T)$.
We now can take $\ell=2$, 
$R ( \overline{S \cup T } ) $ in place of $R$, $V=S_1=S$  
and $W=S_2=T$ in 
equation \eqref{turtt} in Theorem \ref{t},
getting 
$R ( \overline{S \cup T } ) =
R ( \overline{S \cup T } ) (S \circ T)
\subseteq 
R ( \overline{S \cup T } ) 
(\overline{S ^\smallsmile \cup T ^\smallsmile  })
( \overline{S ^\smallsmile  \cup T  }) \circ  (\Theta_R   S+ \Theta_R  T) $,
since  
$( \overline{S \cup T } ) ^\smallsmile =  
\overline{S ^\smallsmile \cup T ^\smallsmile  }$.
Moreover, since 
$S \cup T  = T \cup S$, we can repeat the argument
once again, getting  (D2).
\end{proof}

\begin{remark} \labbel{nondir}
We have made an essential use of the results by 
 Kazda,  Kozik,  McKenzie and Moore \cite{kkmm}  
in order to prove Theorem \ref{t}, hence to prove
equations \eqref{a1}, \eqref{a2} in Corollary \ref{c}
and Condition (C1) in Corollary \ref{tut}.
However, the reader who knows (undirected)
Gumm terms might easily see that the above arguments
can be adapted to get proofs for equation \eqref{a3} in 
\ref{c} and for conditions (A2)-(B2) and (D3)-(D5)
in Corollary \ref{tut} using just Gumm terms.
 This might be convenient
when we want to evaluate the number of actual  factors on the 
right-hand sides, since there might be varieties with a smaller number 
of Gumm terms rather than directed Gumm terms.

In a few cases, it is even 
enough to use just Day terms \cite{D}.
In fact, there is a relation identity which
 characterizes exactly the number
of Day terms of a congruence modular variety. 
See the next proposition.
 \end{remark}   

Recall that \emph{Day terms} are quaternary terms 
$d_0, d_1, \dots, d_k$ satisfying
the following conditions.
\begin{align*}     
 x &=d_i(x,y,y,x) && \text{for every } i;  
\\ 
 x&=d_0(x,y,z,w); 
\\
 d_i(x,x,w,w)&=d_{i+1}(x,x,w,w), && \text{for $i$ even}; 
\\
 d_i(x,y,y,w)&=d_{i+1}(x,y,y,w), && \text{for $i$ odd, and} 
\\ 
 d_{k}(x,y,z,w)&=w.
\end{align*}

\begin{proposition} \labbel{day}
A variety $\mathcal V$ has $k+1$ Day terms
$d_0, d_1, \dots, d_k$ if and only if   
$\mathcal V$ satisfies the identity
\begin{equation*}\labbel{dr}  
\Theta (S \circ S ^\smallsmile ) \subseteq \Theta S \circ_{k-1} \Theta S ^\smallsmile
\end{equation*}    
where  $S$ varies
among reflexive and admissible relations and 
$\Theta$ can be equivalently taken to vary among tolerances 
or among congruences.
 \end{proposition}

  \begin{proof}
Suppose that $\mathcal V$ has Day terms 
 $d_0, d_1, \dots, d_k$ and in some algebra in $\mathcal V$ 
we have 
$(a,c) \in \Theta (S \circ S ^\smallsmile )$,
for $\Theta$ a tolerance.
Thus
$a \mathrel { \Theta } c $
and
there is some $b$ such that
$a \mathrel { S} b   \mathrel { S ^\smallsmile } c$,
hence
$c \mathrel { S } b$. 
Then, say for $k$ even,  
$a=d_1(a,a,c,c) \mathrel S d_1(a,b,b,c)=d_2(a,b,b,c)
\mathrel { S ^\smallsmile } d_2(a,a,c,c) = d_3(a,a,c,c)
\mathrel { S}  d_3(a,b,b,c) \dots  d_{k-1}(a,b,b,c)=c  $.
Moreover, by an argument in Cz\'edli and  Horv\'ath \cite{CH},
\begin{multline*}
d_i(a,a,c,c) = 
d_i(d_j(abb\text{\boldmath$a$}),\text{\boldmath$a$},c,d_j(\text{\boldmath$c$}bbc))
\mathrel { \Theta } 
\\
d_i(d_j(abb\text{\boldmath$c$}),\text{\boldmath$c$},c,d_j(\text{\boldmath$a$}bbc))=
d_j(a,b,b,c),
 \end{multline*}
for every $i,j \leq k$.     
The above relations show that 
 $(a,c) \in  \Theta S \circ_{k-1} \Theta S ^\smallsmile$.

Conversely, suppose that
$\alpha$, $\beta$ and $\gamma$ are congruences and
take
$\Theta = \alpha $ and
$S= \beta \circ \alpha \gamma $
in the identity in the statement of the proposition.
Then   
$\alpha(\beta \circ \alpha \gamma  \circ \beta  )=
\alpha  (S \circ S ^\smallsmile ) \subseteq 
\alpha S \circ_{k-1} \alpha S ^\smallsmile
=
(\alpha (\beta \circ \alpha \gamma) ) \circ_{k-1} ( \alpha (\alpha \gamma \circ \beta ))
=
(\alpha \beta \circ \alpha \gamma ) \circ_{k-1} ( \alpha \gamma \circ \alpha \beta )
=
\alpha \beta \circ_k \alpha \gamma 
$,
since, both $\alpha \beta $ and $\alpha \gamma $ being congruences,
$k-2$ factors are absorbed in the last identity,
hence we end up with exactly $k$ factors.    
It is a standard fact implicit in \cite{D}
that, within a variety, the congruence identity 
$\alpha( \beta \circ \alpha \gamma \circ \beta )
\subseteq \alpha \beta \circ _k \alpha \gamma $
corresponds exactly to the existence of 
$k+1$ Day terms.   
 \end{proof}

For an appropriate value of $r$, 
the identity \eqref{a3} in Corollary \ref{c} can be obtained as a 
consequence of Proposition \ref{day} and, 
according to the respective number of terms in 
some given variety,
we might get a better bound.
Conditions (A2)-(B2) in Corollary \ref{tut}, too,
can be obtained as a consequence of 
Proposition \ref{day}.

\acknowledgement{We thank the students of Tor Vergata
University for  stimulating discussions.}

\smallskip 

This is a preliminary version, still to
be expanded.  It might contain inaccuraccies (to be
precise, it is more likely to contain inaccuracies than  subsequent versions).

  We have not yet performed a completely accurate
search in order to check whether some of the results presented here
are already known. Credits for already known results should go to the
original discoverers.

\smallskip

{\scriptsize
Though the author has done his best efforts to compile the following
list of references in the most accurate way,
 he acknowledges that the list might 
turn out to be incomplete
or partially inaccurate, possibly for reasons not depending on him.
It is not intended that each work in the list
has given equally significant contributions to the discipline.
Henceforth the author disagrees with the use of the list
(even in aggregate forms in combination with similar lists)
in order to determine rankings or other indicators of, e.~g., journals, individuals or
institutions. In particular, the author 
 considers that it is highly  inappropriate, 
and strongly discourages, the use 
(even in partial, preliminary or auxiliary forms)
of indicators extracted from the list in decisions about individuals (especially, job opportunities, career progressions etc.), attributions of funds, and selections or evaluations of research projects.
\par
}

\end{document}